\documentclass[11pt,reqno]{amsart}
\usepackage{amssymb,latexsym}
\usepackage{graphicx}
\usepackage{url}		%does nice formatting of URLs
\usepackage{multicol}

\calclayout

\makeatletter
\newcommand{\monthyear}[1]{%
  \def\@monthyear{\uppercase{#1}}}
\newcommand{\volnumber}[1]{%
  \def\@volnumber{\uppercase{#1}}}
\AtBeginDocument{%
\def\ps@plain{\ps@empty
  \def\@oddfoot{ \hfil \thepage}% \@monthyear
  \def\@evenfoot{\thepage \hfil }} %\@volnumber
\def\ps@firstpage{\ps@plain}
\def\ps@headings{\ps@empty
  \def\@evenhead{%
    \setTrue{runhead}%
    \def\thanks{\protect\thanks@warning}%
    \hfil}% \uppercase{The Fibonacci Quarterly}
  \def\@oddhead{%
    \setTrue{runhead}%
    \def\thanks{\protect\thanks@warning}%
    \hfill\uppercase{Trigonometric-type properties and the parity of balancing ...}}%
  \let\@mkboth\markboth
  \def\@evenfoot{%
    \thepage \hfil }% \@volnumber
  \def\@oddfoot{%
     \hfil \thepage}% \@monthyear
  }%
\footskip=25pt
\pagestyle{headings}%
}
\makeatother

\theoremstyle{definition}
  \newtheorem{df}{Definition}
  
\theoremstyle{plain}
   
   \newtheorem{thm}[df]{Theorem}
   
   \newtheorem{cor}[df]{Corollary}

 \numberwithin{equation}{section}

\begin{document}

\monthyear{Month Year}
\volnumber{Volume, Number}
\setcounter{page}{1}

\title{Trigonometric-type properties and the parity of balancing, cobalancing, Lucas-balancing and Lucas-cobalancing numbers}
\author{Ng{\^o} Van Dinh}
\address{Department of Mathematics and Informatics\\
                Thai Nguyen University of Sciences\\
                Thai Nguyen, Vietnam}
\email{dinh.ngo@tnus.edu.vn}
\thanks{The author gratefully acknowledges support from the Thai Nguyen University of Sciences.}

\begin{abstract}
Balancing numbers $n$ are originally defined as the solution of the Diophantine  equation $1+2+\cdots+(n-1)=(n+1)+\cdots+(n+r)$, where $r$ is called the balancer corresponding to the balancing number $n$. By slightly modifying, $n$ is the cobalancing number with the cobalancer $r$ if $1+2+\cdots+n=(n+1)+\cdots+(n+r)$. Let $B_n$ denote the $n^{th}$ balancing number and $b_n$ denote the $n^{th}$ cobalancing number. Then $8B_n^2+1$ and $8b_n^2+8b_n+1$ are perfect squares. The $n^{th}$ Lucas-balancing number $C_n$ and the $n^{th}$ Lucas-cobalancing number $c_n$ are the positive roots of $8B_n^2+1$ and $8b_n^2+8b_n+1$, respectively. In this paper, we establish some trigonometric-type identities and some arithmetic properties concerning the parity of balancing, cobalancing, Lucas-balancing and Lucas-cobalancing numbers.
\end{abstract}

\maketitle

\section{Introduction}
While studying triangular numbers, Behera and Panda  \cite{BeheraPanda} introduced the notion of \textit{balancing numbers}. An integer $n\in \mathbb{Z}^+$ is a balancing number if
\begin{equation}\label{def_1}
1+2+\cdots+(n-1)=(n+1)+(n+2)+\cdots+(n+r),
\end{equation}
for some $r\in\mathbb{Z}^+$. The number $r$ in (\ref{def_1}) is called the \textit{balancer} corresponding to the balancing number $n$. Behera and Panda also found that $n$ is a balancing number if and only if $n^2$ is a triangular number, as well as, $8n^2+1$ is a perfect square. Though the definition suggests that no balacing number should be less than 2, we accept 1 as a balancing number being the positive square root of the square triangular number 1 \cite{PandaRay2005}. If $n$ is a balancing number then the positive root of $8n^2+1$ is called a \textit{Lucas-balancing number} \cite{Panda2009}.

Let $B_n$ and $C_n$ denote the $n^{th}$ balancing number and the $n^{th}$ Lucas-balancing number, respectively, and set $B_0=0, C_0=1$. Then we have the recurrence relations
\[B_{n+1}=6B_n-B_{n-1}, n\geq 1,\]
with $B_0=0, B_1=1$, and
\[C_{n+1}=6C_n-C_{n-1}, n\geq 1,\]
with $C_0=1, C_1=3$. These recurrence relations give the Binet formulas for balancing and Lucas-balancing numbers as follows:
\[B_1=1, B_2=6, B_n=\frac{\lambda_1^n-\lambda_2^n}{\lambda_1-\lambda_2}, \text{ for all } n\geq 0,\]
and
\[C_1=3, C_2=17, C_n=\frac{\lambda_1^n+\lambda_2^n}{2}, \text{ for all } n\geq 0,\]
where $\lambda_1=3+\sqrt{8}, \lambda_2=3-\sqrt{8}$.

By slightly modifying (\ref{def_1}), Panda and Ray \cite{PandaRay2005} defined \textit{cobalancing numbers} $n\in\mathbb{Z}^+$ as solutions of the Diophantine equation
\begin{equation*}
1+2+\cdots+(n-1)=(n+1)+(n+2)+\cdots+(n+r),
\end{equation*}
where $r$ is called the \textit{cobalancer} corresponding to $n$. An natural number $n$ is a cobalancing number if and only if $8n^2+8n+1$ is a perfect square. So we can accept $0$ is the first cobalancing number. Let $b_n$ be the $n^{th}$ cobalancing number. Then the $n^{th}$ Lucas-cobalancing number $c_n$ is the positive root of $8b_n^2+8b_n+1$. Moreover, we have the recurrence relations \cite{PandaRay2011}
\[b_1=0, b_2=2, b_{n+1}=6b_n-b_{n-1}+2, n\geq 2,\]
and
\[c_1=1, c_2=7, c_{n+1}=6c_n-c_{n-1}, n\geq 2.\]
The following formulas are the Binet forms for cobalancing and Lucas-cobalancing numbers, respectively,
\[b_n=\frac{\alpha_1^{2n-1}-\alpha_2^{2n-1}}{4\sqrt{2}}-\frac{1}{2} \text{ and } c_n=\frac{\alpha_1^{2n-1}+\alpha_2^{2n-1}}{2},\]
where $\alpha_1=1+\sqrt{2}$ and $\alpha_2=1-\sqrt{2}$.

Many interesting properties and important identities of balancing, cobalancing, Lucas-balancing and Lucas-cobalancing numbers are available in the literature.
Panda \cite{Panda2009} established two following identities which look like the trigonometric identities $\sin(x\pm y)=\sin x\cos y \pm \cos x\sin y$:

\begin{equation}\label{Panda1}
\displaystyle B_{n+m}=B_nC_m+B_mC_n
\end{equation}
and
\begin{equation*}
\displaystyle B_{n-m}=B_nC_m-B_mC_n.
\end{equation*}
In this work, we establish some more trigonometric type identities and we deduce from these identities the parity of balancing, cobalancing, Lucas-balancing and Lucas-cobalancing numbers.

%%%%%

\section{Main results}
Starting from Panda's idea, the following theorem give an identity which has the same type of the trigonometric identity $\displaystyle \sin x - \sin y=2\sin (\frac{x-y}{2})\cos(\frac{x+y}{2})$.

\begin{thm}\label{sin-sin-tq}
For $n,m$ are natural numbers such that $n\geq m$ and having the same parity, we have
$$B_{n}-B_{m}=2B_{\frac{n-m}{2}}C_{\frac{n+m}{2}}.$$
\end{thm}

\begin{proof}
Using Binet formulas, we have
\begin{align*}
2B_{\frac{n-m}{2}}C_{\frac{n+m}{2}}&=2.\frac{\lambda_1^{\frac{n-m}{2}}-\lambda_2^{\frac{n-m}{2}}}{\lambda_1-\lambda_2}.\frac{\lambda_1^{\frac{n+m}{2}}+\lambda_2^{\frac{n+m}{2}}}{2}\\
&=\frac{\lambda_1^{n}-\lambda_2^{n}}{\lambda_1-\lambda_2}-\frac{\lambda_1^{m}-\lambda_2^{m}}{\lambda_1-\lambda_2}\\
&=B_{n}-B_{m}.
\end{align*}
This completes the proof.
\end{proof}

%%%%%%%

\begin{cor}\label{sin-sin}
For $n,m$ are natural numbers such that $n\geq m$, we have
$$B_{2n}-B_{2m}=2B_{n-m}C_{n+m}.$$
\end{cor}

\begin{proof}
This is an intermediate consequence of Theorem \ref{sin-sin-tq}.
\end{proof}

In Corollary \ref{sin-sin}, by taking $m=1$ we obtain a corollary of which \cite[Theorem 2.1]{Ray} is a particular case.

%%%%
\begin{cor} For $n\geq 1$, we have
$$B_{2n}-6=2B_{n-1}C_{n+1}.$$
\end{cor}

%%%%%%
\begin{cor} Let $n,m$ be natural numbers such that $n\geq m$. Then
$$B_{2n}=2(B_{n-m}C_{n+m}+B_mC_m).$$
\end{cor}

\begin{proof}
This corollary is an obvious consequence of corollaries \ref{sin-sin} and (\ref{Panda1}).
\end{proof}

%%%%%
With Theorem \ref{sin-sin-tq}, we can see the parity of balancing numbers.

\begin{cor}\label{parity-B_n}
For every integer $n\geq 0$, the balancing number $B_n$ and $n$ have the same parity.
\end{cor}

\begin{proof}
If $n,m$ are integers with the same parity then $B_n$ and $B_m$ have the same parity by Theorem \ref{sin-sin-tq}. On the other one, we have $B_0=0, B_1=1$ and $B_2=6$. It implies that $B_n$ and $n$ have the same parity.
\end{proof}

We also have an identity of balancing numbers which resembles the trigonometric identity $\displaystyle \sin x+\sin y=2\sin(\frac{x+y}{2})\cos(\frac{x-y}{2})$.

%%%%%%%%
\begin{thm}\label{sin+sin_tq}
Let $n,m$ be natural numbers such that $n\geq m$ and having the same parity. Then
$$B_{n}+B_{m}=2B_{\frac{n+m}{2}}C_{\frac{n-m}{2}}.$$
\end{thm}

\begin{proof}
Using Binet formulas, we have
\begin{align*}
2B_{\frac{n+m}{2}}C_{\frac{n-m}{2}}&=2.\frac{\lambda_1^{\frac{n+m}{2}}-\lambda_2^{\frac{n+m}{2}}}{\lambda_1-\lambda_2}.\frac{\lambda_1^{\frac{n-m}{2}}+\lambda_2^{\frac{n-m}{2}}}{2}\\
&=\frac{\lambda_1^{n}-\lambda_2^{n}}{\lambda_1-\lambda_2}+\frac{\lambda_1^{m}-\lambda_2^{m}}{\lambda_1-\lambda_2}\\
&=B_{n}+B_{m}.
\end{align*}
This is what was to be shown.
\end{proof}

%%%%%%%%
\begin{cor}\label{sin+sin}
For $n,m$ are natural numbers such that $n\geq m$, we have
$$B_{2n}+B_{2m}=2B_{n+m}C_{n-m}.$$
\end{cor}

\begin{proof}
This is a direct consequence of Theorem \ref{sin+sin_tq}.
\end{proof}

%%%%%
\begin{cor}
Let $n,m$ be natural numbers such that $n\geq m$. Then
\begin{itemize}
\item[i)] $B_{n-m}C_n+B_nC_{n-m}=B_{2n-m}$;
\item[ii)] $B_nC_{n-m}-B_{n-m}C_n=B_m$.
\end{itemize}
\end{cor}

\begin{proof}
By theorems \ref{sin-sin-tq} and \ref{sin+sin_tq}, we have
\[B_{2n-m}-B_m=2B_{n-m}C_n \text{ and } B_{2n-m}+B_m=2B_{n}C_{n-m}.\]
Hence we obtain the required identities.
\end{proof}

The following theorem shows that we have an identity of Lucas-balancing numbers which looks like the trigonometric identity 
$$\displaystyle\cos x+\cos y=2\cos(\frac{x+y}{2})\cos(\frac{x-y}{2}).$$ 
However, we have another which resembles, up to a scalar, the trigonometric identity 
$$\displaystyle\cos x-\cos y=-2\sin(\frac{x+y}{2})\sin(\frac{x-y}{2}).$$

%%%%%%

\begin{thm}\label{cosa+cosb_tq}
Let $n,m$ be natural numbers such that $n\geq m$ and having the same parity. Then
\begin{itemize}
\item[i)] $C_{n}+C_{m}=2C_{\frac{n+m}{2}}C_{\frac{n-m}{2}}$;
\item[ii)] $C_{n}-C_{m}=16B_{\frac{n+m}{2}}B_{\frac{n-m}{2}}$.
\end{itemize}
\end{thm}

\begin{proof}
Continue using Binet formulas, we have
\begin{align*}
2C_{\frac{n+m}{2}}C_{\frac{n-m}{2}}&=2.\frac{\lambda_1^{\frac{n+m}{2}}+\lambda_2^{\frac{n+m}{2}}}{2}.\frac{\lambda_1^{\frac{n-m}{2}}+\lambda_2^{\frac{n-m}{2}}}{2}\\
&=\frac{\lambda_1^{n}+\lambda_2^{n}}{2}+\frac{\lambda_1^{m}+\lambda_2^{m}}{2}=C_{n}+C_{m}.
\end{align*}
The first identity is proved. To prove the second, we have
\begin{align*}
B_{\frac{n+m}{2}}B_{\frac{n-m}{2}}&=\frac{\lambda_1^{\frac{n+m}{2}}-\lambda_2^{\frac{n+m}{2}}}{\lambda_1-\lambda_2}.\frac{\lambda_1^{\frac{n-m}{2}}-\lambda_2^{\frac{n-m}{2}}}{\lambda_1-\lambda_2}\\
&=\frac{1}{(\lambda_1-\lambda_2)^2}(\lambda_1^{n}+\lambda_2^{n}-\lambda_1^{m}-\lambda_2^{m})\\
&=\frac{1}{32}(\lambda_1^{n}+\lambda_2^{n}-\lambda_1^{m}-\lambda_2^{m})=\frac{1}{16}(C_{n}-C_{m}).
\end{align*}
This implies the required identity.
\end{proof}

%%%%%%

\begin{cor}\label{cosa+cosb}
For $n,m$ are natural numbers such that $n\geq m$, we have
\begin{itemize}
\item[i)] $C_{2n}+C_{2m}=2C_{n+m}C_{n-m}$;
\item[ii)] $C_{2n}-C_{2m}=16B_{n+m}B_{n-m}$.
\end{itemize}
\end{cor}

\begin{proof}
These identities directly follow from Theorem \ref{cosa+cosb_tq}.
\end{proof}

%%%%%%
\begin{cor}
Let $n,m$ be natural numbers such that $n\geq m$. Then
\begin{itemize}
\item[i)] $C_nC_{n-m}+8B_nB_{n-m}=C_{2n-m}$;
\item[ii)] $C_nC_{n-m}-8B_nB_{n-m}=C_{m}$.
\end{itemize}
\end{cor}

\begin{proof}
By Theorem \ref{cosa+cosb_tq}, we have
\[C_nC_{n-m}=\frac{1}{2}(C_{2n-m}+C_m) \text{ and } 8B_nB_{n-m}=\frac{1}{2}(C_{2n-m}-C_m).\]
These identities imply the required identities.
\end{proof}

%%%%%%
Now, we can see the parity of Lucas-balancing numbers.

\begin{cor}
For all integer $n\geq 0$, the Lucas-balancing number $C_n$ is odd. Moreover, if $n,m$ are integers with the same parity then the difference between $C_n$ and $C_m$ is divisible by 16.
\end{cor}

\begin{proof}
If $n,m$ are integers with the same parity then the difference between $C_n$ and $C_m$ is divisible by 16 by the second identity of Theorem \ref{cosa+cosb_tq}. This also means that $C_n$ and $C_m$ have the same parity. On the other hand, we have $C_0~=~1$,$C_1=3$, $C_2=17$. It implies that $C_n$ is odd for all $n$.
\end{proof}

We can not find identities for Lucas-balancing numbers which resemble the trigonometric identities $\displaystyle \cos(x\pm y)=\cos x\cos y \mp \sin x\sin y$. But we establish the following interesting theorem.

%%%%%%

\begin{thm}\label{cos(a+b)}
Let $n,m$ be natural numbers such that $n\geq m$. Then
\begin{itemize}
\item[i)] $16(C_nC_m-B_nB_m)=7C_{n+m}+9C_{n-m}$;
\item[ii)] $16(C_nC_m+B_nB_m)=9C_{n+m}+7C_{n-m}$.
\end{itemize}
\end{thm}

\begin{proof}
Applying Binet forms, we have
\begin{align*}
C_nC_m&-B_nB_m=\frac{\lambda_1^n+\lambda_2^n}{2}.\frac{\lambda_1^m+\lambda_2^m}{2}-\frac{\lambda_1^n-\lambda_2^n}{\lambda_1-\lambda_2}.\frac{\lambda_1^m-\lambda_2^m}{\lambda_1-\lambda_2}\\
&=\frac{\lambda_1^{n+m}+\lambda_2^{n+m}+\lambda_1^{n-m}+\lambda_2^{n-m}}{4}-\frac{\lambda_1^{n+m}+\lambda_2^{n+m}-\lambda_1^{n-m}-\lambda_2^{n-m}}{32}\\
&=\frac{7}{16}.\frac{\lambda_1^{n+m}+\lambda_2^{n+m}}{2}+\frac{9}{16}.\frac{\lambda_1^{n-m}+\lambda_2^{n-m}}{2}=\frac{7C_{n+m}+9C_{n-m}}{16}.
\end{align*}
Hence we get the first identity. The second is proved by similar calculations.
\end{proof}

%%%%%%%%
Motived by the above results, we establish some identities of balancing, cobalancing, Lucas-balancing and Lucas-cobalancing numbers. We also obtain some properties on the parity of these numbers. The following theorem give us relations between sums of Lucas-balancing numbers and Lucas-cobalancing numbers or cobalancing numbers.

\begin{thm}\label{c_nc_m, b_nb_m}
For $n,m$ are integers such that $n\geq m\geq 1$, we have
\begin{itemize}
\item[i)] $C_{n+m-1}-C_{n-m}=2c_nc_m$;
\item[ii)] $C_{n+m-1}+C_{n-m}=16b_nb_m+8(b_n+b_m)+4$.
\end{itemize}
\end{thm}

\begin{proof}
Using Binet forms with remark that $\alpha_1\alpha_2=-1$, we have
\begin{align*}
c_nc_m&=\frac{\alpha_1^{2n-1}+\alpha_2^{2n-1}}{2}.\frac{\alpha_1^{2m-1}+\alpha_2^{2m-1}}{2}\\
&=\frac{\alpha_1^{2(n+m-1)}+\alpha_2^{2(n+m-1)}}{2}-\frac{\alpha_1^{2(n-m)}+\alpha_2^{2(n-m)}}{2}\\
&=\frac{1}{2}(C_{n+m-1}-C_{n-m}).
\end{align*}
Hence we obtain the first identity. Similarly, we can prove the second identity.
\end{proof}

By ii) of Theorem \ref{c_nc_m, b_nb_m}, we have the following consequence about sum of two consecutive Lucas-balancing numbers.

\begin{cor}
For all integer $n\geq 1$, the sum of $(n-1)^{th}$ and $n^{th}$ Lucas-balancing numbers is divisible by 4.
\end{cor}

The following theorem is an interesting property of sums of two cobalancing numbers from which we can see the parity of cobalancing numbers.

%%%%%%
\begin{thm}\label{b_n-b_m}
Let $n,m$ be positive integers.
\begin{itemize}
\item[i)] If $n>m$ then $b_{n+m}-b_{n-m}=2c_nB_m$;
\item[ii)] If $n\leq m$ then $b_{n+m}-b_{m-n+1}=2c_nB_m$.
\end{itemize}
\end{thm}

\begin{proof}
By Binet formulas, we have
\begin{align*} 
c_nB_m&=\frac{\alpha_1^{2n-1}+\alpha_2^{2n-1}}{2}.\frac{\alpha_1^{2m}-\alpha_2^{2m}}{4\sqrt{2}}\\
&=\frac{\alpha_1^{2(n+m)-1}-\alpha_2^{2(n+m)-1}}{8\sqrt{2}}-\frac{\alpha_1^{2(n-m)-1}-\alpha_2^{2(n-m)-1}}{8\sqrt{2}}\\
&=\begin{cases} \displaystyle\frac{1}{2}(b_{n+m}-b_{n-m}), \text { if } n>m;\\ \displaystyle \frac{1}{2}(b_{n+m}-b_{m-n+1}), \text { otherwise.}\end{cases}
\end{align*}
Hence we have what was to be demonstrated. 
\end{proof}

\begin{cor}
The cobalancing numbers are even. Moreover, for all $m\geq 1$, the difference between the $(2m+1)^{th}$ and $(2m)^{th}$ cobalancing numbers  is divisible by 4.
\end{cor}

\begin{proof}
By ii) of Theorem \ref{b_n-b_m}, we can see that $b_n$ and $b_{n+1}$ have the same parity for all $n\geq 1$. It follows that $b_n$ is even for all $n\geq 1$ since $b_1=0$. Moreover, from ii) of Theorem \ref{b_n-b_m}, we also obtain the second affirmation since $B_{2m}$ is even by Corollary \ref{parity-B_n}.
\end{proof}

The following theorem is another property of sums of two cobalancing numbers.

%%%%%%
\begin{thm}
Let $n,m$ be positive integers.
\begin{itemize}
\item[i)] If $n>m$ then $b_{n+m}+b_{n-m}=2b_nC_m+C_m-1$;
\item[ii)] If $n\leq m$ then $b_{n+m}-b_{m-n+1}=2b_nC_m+C_m-1$.
\end{itemize}
\end{thm}

\begin{proof}
By Binet forms, we have
\begin{align*}
b_nC_m&=\left(\frac{\alpha_1^{2n-1}-\alpha_2^{2n-1}}{4\sqrt{2}}-\frac{1}{2}\right).\frac{\alpha_1^{2m}+\alpha_2^{2m}}{2}\\
&=\frac{\alpha_1^{2(n+m)-1}-\alpha_2^{2(n+m)-1}}{8\sqrt{2}}+\frac{\alpha_1^{2(n-m)-1}-\alpha_2^{2(n-m)-1}}{8\sqrt{2}}-\frac{1}{2}C_m\\
&=\begin{cases}\displaystyle\frac{1}{2}(b_{n+m}+b_{n-m}-C_m+1), \text { if } n>m;\\ \displaystyle\frac{1}{2}(b_{n+m}+b_{m-n+1}-C_m+1), \text { otherwise.}\end{cases}
\end{align*}
Hence we obtain the required identities.
\end{proof}

In the following theorem, we again get an identity of cobalancing and Lucas-cobalancing numbers which looks like the trigonometric identity $\displaystyle \sin x+\sin y=2\sin(\frac{x+y}{2})\cos(\frac{x-y}{2})$.

%%%%%
\begin{thm}\label{c_n-c_m}
Let $n,m$ be positive integers.
\begin{itemize}
\item[i)] If $n>m$ then $c_{n+m}+c_{n-m}=2c_nC_m$;
\item[ii)] If $n\leq m$ then $c_{n+m}-c_{m-n+1}=2c_nC_m$.
\end{itemize}
\end{thm}

\begin{proof}
By Binet forms, we have
\begin{align*}
c_nC_m&=\frac{\alpha_1^{2n-1}+\alpha_2^{2n-1}}{2}.\frac{\alpha_1^{2m}+\alpha_2^{2m}}{2}\\
&=\frac{\alpha_1^{2(n+m)-1}+\alpha_2^{2(n+m)-1}}{4}+\frac{\alpha_1^{2(n-m)-1}+\alpha_2^{2(n-m)-1}}{4}\\
&=\begin{cases}\displaystyle\frac{1}{2}(c_{n+m}+c_{n-m}), \text{ if } n>m;\\ \displaystyle\frac{1}{2}(c_{n+m}-c_{m-n+1}), \text{ otherwise.}\end{cases}
\end{align*}
This completes the proof.
\end{proof}

We can deduce the parity of the Lucas-cobalancing numbers from the second identity of the previous theorem.

\begin{cor}
The Lucas-cobalancing numbers are odd.
\end{cor}

\begin{proof}
By ii) of Theorem \ref{c_n-c_m}, we can see that $c_n$ and $c_{n+1}$ have the same parity for all $n\geq 1$. It follows that $c_n$ is odd for all $n\geq 1$ since $c_1=1$.
\end{proof}

The following theorem give us a better property on the parity of Lucas-cobalancing numbers of even index. It shows that the  $(2n)^{th}$ Lucas-cobalancing number is congruent to $-1$ modulo $8$ and the  $(4n)^{th}$ Lucas-cobalancing number is congruent to $-1$ modulo $16$, for all $n\geq 1$.

%%%%%%
\begin{thm}
For integer $n\geq 1$, we have
\[c_{2n}+1=8(2b_n+1)B_n.\]
\end{thm}

\begin{proof}
By Binet forms, we have
\begin{align*}
b_nB_n&=\left(\frac{\alpha_1^{2n-1}-\alpha_2^{2n-1}}{4\sqrt{2}}-\frac{1}{2}\right).\frac{\alpha_1^{2n}-\alpha_2^{2n}}{4\sqrt{2}}\\
&=\frac{\alpha_1^{4n-1}+\alpha_2^{4n-1}}{32}-\frac{\alpha_1^{-1}+\alpha_2^{-1}}{32}-\frac{1}{2}B_n\\
&=\frac{1}{16}(c_{2n}+1)-\frac{1}{2}B_n.
\end{align*}
Hence we obtain the required identity.
\end{proof}

%======================================

%}

\bigskip
\hrule
\bigskip

\noindent MSC2010:
11B37, 11B83.

\noindent \emph{Keywords: }
Balancing number, Lucas-balancing number, cobalancing number, Lucas-cobalancing number, recurrence relation, the parity, trigonometric-type identity

\end{document}